\documentclass{my_aims}

\def\typeofarticle{Research article}

\def\ppages{4552--4571}
\def\DOI{10.3934/mbe.2021231}
\def\Received{March 2021}
\def\Accepted{May 2021}
\def\Published{May 2021}

\usepackage[pagewise]{lineno}
\usepackage{bbm}
\usepackage{cite}

\numberwithin{equation}{section}

\DeclareMathOperator{\R}{\mathbbm{R}}
\DeclareMathOperator{\e}{e}

\makeatletter
\renewcommand{\@biblabel}[1]{#1\hfill \hspace{-0.2cm}}
\makeatother

\newtheorem{theorem}{Theorem}

\newtheorem{lemma}{Lemma}
\newtheorem{remark}{Remark}
\newtheorem{proposition}{Proposition}
\usepackage{subcaption}

\usepackage{graphicx}

\begin{document}

\title{A dynamically-consistent nonstandard finite difference scheme for the SICA model}

\author{Sandra Vaz\affil{1,}\corrauth
and 
Delfim F. M. Torres\affil{2}}

\shortauthors{Author(s)}

\address{%
\addr{\affilnum{1}}{Center of Mathematics and Applications (CMA-UBI),
Department of Mathematics, University of Beira Interior, Covilh\~{a} 6201-001, Portugal}
\addr{\affilnum{2}}{Center for Research and Development in Mathematics and Applications (CIDMA),
Department of Mathematics, University of Aveiro, Aveiro 3810-193, Portugal}}

\corraddr{Email: svaz@ubi.pt; Tel: +351275242091.}


\begin{abstract}
In this work, we derive a nonstandard finite difference scheme for the SICA 
(Susceptible--Infected--Chronic--AIDS)
model and analyze the dynamical properties 
of the discretized system. We prove that the discretized model is dynamically 
consistent with the continuous, maintaining the essential properties of the 
standard SICA model, namely, the positivity and boundedness of the solutions, 
equilibrium points, and their local and global stability.
\end{abstract}


\maketitle


\section{Introduction}

The Human Immunodeficiency Virus (HIV) is responsible 
for a very high number of deaths worldwide. 
Acquired ImmunoDeficiency Syndrome (AIDS) is a disease 
of the human immune system caused by infection with HIV. 
The HIV virus can be transmitted by several 
ways but there is no cure or vaccine for AIDS. 
Nevertheless, antiretroviral (ART) treatment improves health, 
prolongs life and reduces the risk of HIV transmission. 
The ART treatment increases life expectation but has some limitations. 
For instance, it doesn't restore health, has some side effects, and is very expensive. 
Individuals infected with HIV are more likely to develop tuberculosis because 
of their immunodeficiency, so a model that considers HIV and tuberculosis 
is very interesting to investigate. One can find many studies 
in the literature \cite{Sal1,Sal2,Elaiw}. A TB-HIV/AIDS co-infection model, 
that contains the celebrated SICA (Susceptible--Infected--Chronic--AIDS)
model as a sub-model, was first proposed in \cite{MyID:318}:
\begin{equation}
\label{eq:model}
\begin{cases}
\dot{S}(t) = \Lambda - \lambda (t) S(t) - \mu S(t),\\[0.2 cm]
\dot{I}(t) = \lambda(t) S(t) - (\rho + \phi + \mu)I(t)
+ \alpha A(t)  + \omega C(t), \\[0.2 cm]
\dot{C}(t) = \phi I(t) - (\omega + \mu)C(t),\\[0.2 cm]
\dot{A}(t) =  \rho \, I(t) - (\alpha + \mu + d) A(t),
\end{cases}
\end{equation}
where
\begin{equation}
\label{eq:lambda1}
\lambda(t) = \frac{\beta}{N(t)} \left[ I(t) + \eta_C C(t) + \eta_A A(t) \right]
\end{equation}
with
\begin{equation}
\label{pop}
N(t) = S(t) + I(t) + C(t) + A(t)
\end{equation}
the total population at time $t$.
The meaning of the parameters that appear in the SICA model
\eqref{eq:model}--\eqref{pop} are given in Table~\ref{table:parameters}.
\begin{table}[H]
\centering
\setlength{\tabcolsep}{20mm}
\caption{Parameters of the SICA model \eqref{eq:model}--\eqref{pop}.}
\begin{tabular}{l p{6.5cm}} \hline
{{Symbol}} &  {{Description}} \\ \hline
{{$N(0)$}} & {{Initial population}} \\
{{$\Lambda$}} & {{Recruitment rate}} \\
{{$\mu$}} & {{Natural death rate}} \\
{{$\beta$}} & {{HIV transmission rate}} \\
{{$\phi$}} & {{HIV treatment rate for $I$ individuals}} \\
{{$\rho$}} & {{Default treatment rate for $I$ individuals}}\\
{{$\alpha$}} & {{AIDS treatment rate}}\\
{{$\omega$}} & {{Default treatment rate for $C$ individuals}}\\
{{$d$}} & {{AIDS induced death rate}} \\
{{$\eta_C$}} & {{Modification parameter}} \\
{{$\eta_A$}} & {{Modification parameter}} \\ \hline 
\end{tabular}
\label{table:parameters}
\end{table}
The model considers a varying population size 
in a homogeneously mixing population, subdividing 
the human population into four mutually-exclusive 
compartments:
\begin{itemize}
\item[-] susceptible individuals ($S$);
\item[-] HIV-infected individuals with no clinical symptoms of AIDS 
(the virus is living or developing in the individuals 
but without producing symptoms or only mild ones) 
but able to transmit HIV to other individuals ($I$); 
\item[-] HIV-infected individuals under ART treatment (the so called 
chronic stage) with a viral load remaining low ($C$); 
\item[-] HIV-infected individuals with AIDS clinical symptoms ($A$).
\end{itemize}  
The SICA model has some assumptions. It can be seen in \cite{MyID:455} 
that the susceptible population is increased by the recruitment of individuals 
into the population, assumed susceptible at a rate $\Lambda$. All individuals 
suffer from natural death at a constant rate $\mu$. Susceptible individuals $S$ 
acquire HIV infection, following effective contact with people infected with HIV, 
at rate $\lambda$ \eqref{eq:lambda1}, where $\beta$ is the effective contact rate 
for HIV transmission. The modification parameter $\eta_{A} \geq 1$ accounts for 
the relative infectiousness of individuals with AIDS symptoms, in comparison 
to those infected with HIV with no AIDS symptoms. Individuals with AIDS symptoms 
are more infectious than HIV-infected individuals because they have a higher  viral 
load and there is a positive correlation between viral load and infectiousness. 
On the other hand, $\eta_{C} \leq 1$ translates the partial restoration of the 
immune function of individuals with HIV infection that use correctly ART.
The SICA mathematical model \eqref{eq:model}--\eqref{pop} is well-studied 
in the literature \cite{MyID:455}. It has shown to provide a proper description 
with respect to the HIV/AIDS situation in Cape Verde \cite{MyID:359}
and recent extensions include stochastic transmission \cite{MyID:406}
and fractional versions with memory and general incidence rates \cite{MyID:471}. 
Here our main aim is to propose, for the first time in the literature, 
a discrete-time SICA model.

For most nonlinear continuous models in engineering and natural sciences, 
it is not possible to obtain an exact solution \cite{MyID:429}, so a 
variety of methods have been constructed to compute 
numerical solutions \cite{MyID:446,MyID:447}. It is well known 
that numerical methods, like the Euler and Runge--Kutta, among others, 
often fail to solve nonlinear systems. One of the reasons is that they generate 
oscillations and unsteady states if the time step size decreases 
to a critical size \cite{Stuart:98}. Among available approaches to address the problem, 
the nonstandard finite discrete difference (NSFD) schemes, 
introduced by Mickens in \cite{Mickens:94,Mickens:02}, have been 
successfully applied to several different epidemiological models \cite{Art:1,Art:2}.   
Precisely, the NSFD schemes were created to eliminate or reduce the occurrence 
of numerical instabilities that generally arise while using other methods. 
This is possible because there are some designed laws that systems must satisfy 
in order to preserve the qualitative properties of the continuous model, such as positivity, 
boundedness, stability of the equilibrium points, conservation laws, and others \cite{Mickens:05}.
The literature on Mickens-type NSFD schemes is now vast \cite{MR4183276,MR4128034}.
The paper \cite{MR4141413} considers the NSFD method of Mickens and apply it to a dynamical system 
that models the Ebola virus disease. In \cite{MR3093413}, a NSFD scheme is designed 
in which the Metzler matrix structure of the continuous model is carefully incorporated 
and both Mickens' rules on the denominator of the discrete derivative and the nonlocal 
approximation of nonlinear terms are used. In that work the general analysis is detailed 
for a MSEIR model. In \cite{MR3575285}, the authors summarize NSFD methods  
and compare their performance for various step-sizes when applied to a specific two-sex (male/female) 
epidemic model; while in \cite{MR3316778} it is shown that Mickens' approach is qualitatively superior  
to the standard approach in constructing numerical methods with respect to productive-destructive 
systems (PDS's). NSFD schemes for PDS's  are also investigated in \cite{MR2854820}; 
NSFD methods for predator-prey models with the Beddington--De Angelis functional response 
are studied in  \cite{MR2316130}. Here we propose and investigate, for the first 
time in the literature, the dynamics of a discretized SICA model using 
the Mickens NSFD scheme.

The paper is organized as follows. Some considerations, regarding the continuous SICA model 
and the stability of discrete-time systems, are presented in section~\ref{sec2}. 
The original results are then given in section~\ref{sec3}: we start by introducing 
the discretized SICA model; we find the equilibrium points, 
prove the positivity, Theorem~\ref{positiv}, and boundedness of the solutions, 
Theorem~\ref{conservation}; we also establish the local stability 
of the disease free equilibrium point of the discrete model, Theorem \ref{local}, as well as 
the global stability of the equilibrium points, Theorems~\ref{globalDFE} and \ref{globalEE}. 
In section~\ref{sec4}, we provide some numerical simulations to illustrate the stability 
of the NSFD discrete SICA model using a case study. 
We end with section~\ref{sec5} of conclusion.


\section{Preliminaries}
\label{sec2}

In this section, we collect some preliminary results 
about the continuous SICA model \cite{MyID:318}, 
as well as results for the stability of discrete-time 
systems \cite{Elaydi:05}, useful in our work.


\subsection{The continuous SICA model}

All the information in this section is proved in \cite{MyID:318}. 
Each solution $(S(t), I(t), C(t), A(t))$ of the continuous  model 
much satisfy $S(0) \geq 0$, $I(0) \geq 0$, $C(0) \geq 0$, and $A(0) \geq 0$, 
because each equation represents groups of human beings. 
Adding the four equations of \eqref{eq:model}, one has
$$
\dfrac{d N}{d t}= \Lambda - \mu N- d A \leq \Lambda -\mu N,
$$
so that
$$
N(t) \leq \dfrac{\Lambda}{\mu} + \left(N_{0}-\dfrac{\Lambda}{\mu}\right)\e^{- \mu t}.
$$
Therefore, the biologically feasible region  is given by 
\begin{equation}
\label{eq:omega}
\Omega=\left\{ (S,I,C,A) \in (\R_0^4)^{+}:0 \leq S+I+C+A \leq \frac{\Lambda}{\mu} \right\},
\end{equation}
which is positively invariant and compact. This means that it is sufficient 
to study the qualitative dynamics in $\Omega$. The model has two equilibrium points:
a disease free and an endemic one. The disease free equilibrium (DFE) point is given by 
$$
(S^{\ast},I^{\ast},C^{\ast},A^{\ast})=\left( \frac{\Lambda}{ \mu}, 0, 0, 0\right).
$$
Following the approach of the next generation matrix \cite{Driessche:02}, 
the basic reproduction number $R_{0}$ for model \eqref{eq:model}, 
which represents the expected average number of new infections 
produced by a single HIV-infected individual in contact 
with a completely susceptible population, is given by 
$$
R_{0}=\dfrac{\beta (C_{3} C_{2} +\rho \eta_{A} C_{3} 
+ \phi  \eta_{C} C_{2})}{ \rho C_{3}  (\mu + d) +\mu C_{2} (C_{3} + \phi)}
=:\dfrac{\mathcal{N}}{\mathcal{D}},
$$
where along all the manuscript we use $C_{1}=\rho + \phi + \mu$, 
$C_{2}= \alpha + \mu+ d$, and $C_{3}= \omega + \mu$. The endemic 
point has the following expression: 
$$
(S^{\ast},I^{\ast},C^{\ast},A^{\ast})
=\left(\frac{\Lambda}{\lambda^{\ast} +\mu},  
-\frac{\lambda^{\ast} \Lambda C_{2}C_{3}}{D}, 
-\frac{\phi \lambda^{\ast} \Lambda C_{2}}{D}, 
-\frac{\rho \lambda^{\ast} \Lambda C_{3}}{D} \right),
$$ 
where $D=-(\lambda^{\ast} +\mu) (\mu (C_{3}(\rho 
+ C_{2}) +C_{2} \phi + \rho d) + \rho \omega d)$ and 
\begin{equation}
\label{lambda}
\lambda^{\ast}= \dfrac{\beta( I^{\ast}+\eta_{C} C^{\ast} 
+ \eta_{A} A^{\ast})}{N^{\ast}}
=\dfrac{\mathcal{D}(\mathcal{R}_{0}-1)}{C_{2}C_{3}+\phi C_{2}+\rho C_{3}},
\end{equation}
which is positive if $\mathcal{R}_{0}>1$. The explicit expression 
of the endemic equilibrium point of \eqref{eq:model} is given by
\begin{equation*}
S^{\ast}=\dfrac{\Lambda (\mathcal{D}
-\rho d C_{3})}{\mu (\mathcal{N} - \rho d C_{3})},
\quad
I^{\ast}=\dfrac{\Lambda C_{2} C_{3} (\mathcal{D}
-\mathcal{N})}{\mathcal{D}(\rho d C_{3} -\mathcal{N})},
\quad 
C^{\ast}=\dfrac{\Lambda C_{2} \phi (\mathcal{D}
-\mathcal{N})}{\mathcal{D}(\rho d C_{3} -\mathcal{N})},
\quad
A^{\ast}=\dfrac{\Lambda \rho C_{3} (\mathcal{D}
-\mathcal{N})}{\mathcal{D}(\rho d C_{3} -\mathcal{N})}.
\end{equation*}
Regarding the stability of the equilibrium points, 
Theorem~3.1 and Proposition~3.4 of \cite{MyID:318} 
establish the persistence of the endemic point. 
The disease is  persistent in the population if the 
infected cases with AIDS are bounded away from zero or the population $S$ disappears. 
The local stability of the endemic point is given in Theorem~3.8 of \cite{MyID:318}. 
Lemma~3.5 of \cite{MyID:318} states that the DFE is locally asymptotically stable 
if $R_{0} <1$ and unstable if $R_{0}>1$. Finally, Theorem~3.6 of \cite{MyID:318}
asserts that, under suitable conditions, the DFE point is globally asymptotically stable.
Here we prove similar properties in the discrete-time setting (section~\ref{sec3}).
For that we now recall an important tool for difference equations.


\subsection{The Schur--Cohn criterion}

One of the main tools that provides necessary and sufficient conditions 
for the zeros of a $n$th-degree polynomial
\begin{equation}
\label{eq:polincar}
p(\lambda)= \lambda^{k}+p_{1}\lambda^{k-1}+\cdots+p_{k}
\end{equation}
to lie inside the unit disk is the Schur--Cohn criterion \cite{Elaydi:05}. 
This result is useful for studying the stability of the zero solution of a $k$th-order 
difference equation or to investigate the stability of a $k$-dimensional system of the form
$$
x(n+1)=A x(n),
$$
where $p(\lambda)$ in \eqref{eq:polincar} is the characteristic polynomial of the matrix $A$. 
Let us introduce some preliminaries before presenting the Schur--Cohn criterion. 
Namely, let us define the inners of a matrix $B=(b_{ij})$. The inners of a matrix are the matrix 
itself and all the matrices obtained by omitting successively the first and last columns 
and first and last rows. A matrix $B$ is said to be \emph{positive innerwise} 
if the determinant of all its inners are positive. 

\begin{theorem}[The Shur--Cohn criterion \cite{Elaydi:05}]
\label{jury}
The zeros of the characteristic polynomial \eqref{eq:polincar} 
lie inside the unit disk if, and only if, the following holds:
\begin{enumerate}
\item[i)] $p(1)>0$;
\item[ii)] $(-1)^{k}p(-1)>0$;
\item[iii)] the $(k-1) \times (k-1)$ matrices
$$
B_{k-1}^{\pm}=
\left(
\begin{array}{ccccc}
1  & 0 & \cdots& & 0  \\
p_{1}   & 1 & \cdots & &0   \\
\vdots &  & & & \vdots \\
p_{k-3}&&& & p_{1}\\
p_{k-2}& p_{k-3}&\cdots&p_{1}&1  \\
\end{array}
\right)
\pm
\left(
\begin{array}{ccccc}
0&0   &\cdots &0&p_{k}   \\
0 & 0  &\cdots&p_{k}&p_{k-1}   \\
\vdots &\vdots  & &   &\vdots\\
0& p_{k}& & &p_{3}\\
p_{k}&p_{k-1}&\cdots&p_{3}&p_{2}\\
\end{array}
\right)
$$
are positive innerwise.
\end{enumerate}
\end{theorem}

Using the Schur--Cohn criterion, one may obtain necessary and sufficient conditions 
on the $p_{i}$ coefficients such that the zero solution of \eqref{eq:polincar} 
is locally asymptotically stable. 

\begin{theorem}[See \cite{Elaydi:05}]
Let $A \in \mathcal{M}_{k \times k}$ and let \eqref{eq:polincar} 
be its characteristic polynomial. Then, $p(\lambda)$ is a polynomial of degree $k$. 
Moreover, $p(\lambda)$ has the form
\begin{equation}
p(\lambda)= (-1)^{k}\lambda^{k}+ (-1)^{k-1} Tr(A) \lambda^{k-1}+ \cdots +\det(A).
\end{equation}
\end{theorem}


\section{Main results}
\label{sec3}

We begin by proposing a discrete-time SICA model.


\subsection{The NSFD scheme}
\label{sec31}

One of the important features of the discrete-time epidemic models 
obtained by Mickens method is that they present the same features 
as the corresponding original continuous-time models. Here, we construct 
a dynamically consistent numerical NSFD scheme for solving 
\eqref{eq:model} based on \cite{Mickens:94,Mickens:02,Mickens:05}.  
Let us  define the time instants $t_{n}= nh$ with $n$ integer, 
$h=t_{n+1} -t_{n}$ as the time step size, and $(S_{n}, I_{n}, C_{n}, A_{n})$ 
as the approximated values of $(S(nh), I(nh), C(nh), A(nh))$. Thus, 
the NSFD scheme for model \eqref{eq:model} takes the following form:
\begin{equation}
\label{eq:discmodel}
\begin{cases}
\dfrac{S_{n+1}-S_{n}}{\psi(h)} = \Lambda - \tilde{\lambda}_{n} S_{n+1} - \mu S_{n+1},\\[0.2 cm]
\dfrac{I_{n+1}-I_{n}}{\psi(h)} = \tilde{\lambda}_{n} S_{n+1} - (\rho + \phi + \mu)I_{n+1}
+ \alpha A_{n+1}  + \omega C_{n+1}, \\[0.2 cm]
\dfrac{C_{n+1}-C_{n}}{\psi(h)} = \phi I_{n+1} - (\omega + \mu)C_{n+1},\\[0.2 cm]
\dfrac{A_{n+1}-A_{n}}{\psi(h)} =  \rho \, I_{n+1} - (\alpha + \mu + d) A_{n+1},
\end{cases}
\end{equation}
where $ \tilde{\lambda}_{n}=\frac{\beta}{N_{n}}(I_{n}+\eta_{C}C_{n}+\eta_{A}A_{n})$.
The nonstandard schemes are based in two fundamental principles \cite{Mickens:02,Mickens:05}:
\begin{enumerate}
\item Regarding the first derivative, we have
$$
\dfrac{dx}{dt} \rightarrow \dfrac{x_{k+1}-\nu(h) x_{k}}{\psi(h)},
$$ 
where $\nu(h)$ and $\psi(h)$ are the numerator and denominator 
functions that satisfy the requirements
$$
\nu(h)=1+O(h^{2}), \qquad \psi(h)=h+ O(h^{2}).
$$
In general, the numerator function can be selected to be $\nu(h)=1$. 
We will make this choice here. Generally, the denominator function 
is nontrivial. Based on Mickens work \cite{Mickens:07,Mickens:13}, 
when we write explicitly $S_{n+1}$ using $\phi(h)=h$, 
we have in the denominator the term $1+\mu h$, 
which means that we can use $\psi(h)=\dfrac{\e^{\mu h}-1}{\mu}$  
as the denominator function. Throughout our study, 
$\psi(h)=\dfrac{\e^{\mu h}-1}{\mu}$ but,
for brevity, we write $\psi(h)=\psi$.

\item Both linear and nonlinear functions of $x(t)$ and its derivatives 
may require a ``nonlocal'' discretization. For example, 
$x^{2}$ can be replaced by $x_{k} x_{k+1}$.
\end{enumerate}

Since model \eqref{eq:discmodel} is linear in $S_{n+1}$, $I_{n+1}$,
$C_{n+1}$, and $A_{n+1}$, we can obtain their explicit form:
\begin{equation}
\label{eq:discmodelexplicit}
\begin{cases}
S_{n+1}=\dfrac{S_{n}+\Lambda \psi}{1+\mu \psi +\frac{\beta \psi}{N_{n}} 
\left(I_{n}+ \eta_{C} C_{n}+ \eta_{A} A_{n}\right)},\\[0.3cm]
I_{n+1}=\dfrac{\left(I_{n}+\frac{\beta \psi S_{n+1}(I_{n}
+\eta_{C}C_{n}+\eta_{A}A_{n})}{N_{n}}\right)(1+C_{2}\psi)(1+C_{3} \psi)
+\alpha \psi A_{n}(1+C_{3}\psi)+\omega \phi C_{n}(1+C_{2} \psi)}{(1
+C_{1}\psi)(1+C_{2}\psi)(1+C_{3}\psi)
-\alpha\rho \psi^{2}(1+C_{3} \psi)
-\omega \phi \psi^{2}(1+C_{2}\psi)},\\[0.3cm]
C_{n+1}=\dfrac{\phi \psi I_{n+1}+C_{n}}{1+C_{3}\psi},\\[0.3cm]
A_{n+1}=\dfrac{\rho \psi I_{n+1} + A_{n}}{1+C_{2} \psi}.
\end{cases}
\end{equation}


\subsection{Positivity of solutions}
\label{subsec:post:sol}

The first result to be shown is the positivity of the solutions.

\begin{theorem}
\label{positiv}
If all the initial and parameter values of the discrete system 
\eqref{eq:discmodelexplicit} are positive, then the
solutions are always positive for all $n \geq 0$ 
with denominator function $\psi$.
\end{theorem}

\begin{proof}
Let us assume that $S(0), I(0), C(0), A(0)$ are positive.  
We only need to show that $I_{n+1}$ is positive.
The denominator is given by
\begin{equation}
(1+C_{1}\psi)(1+C_{2}\psi)(1+C_{3}\psi)
-\alpha\rho \psi^{2}(1+C_{3} \psi)-\omega \phi \psi^{2}(1+C_{2}\psi),
\end{equation}
which can be rewritten as
\begin{equation*}
1+(C_{1}+C_{2} +C_{3})\psi  + (C_{1} C_{2} 
+ C_{1}C_{3}+ C_{2} C_{3}-\phi \omega -\alpha \rho)\psi^{2} 
+ (C_{1} C_{2} C_{3} -C_{2} \phi \omega -\alpha \rho C_{3})\psi^{3}
\end{equation*}
and, simplifying, we get
\begin{multline*}
1+(C_{1}+C_{2} +C_{3})\psi +\left(
C_{2}(2\mu+\phi+\omega)+C_{3}(\omega+\rho)+\mu(\phi+\rho)+\rho d\right)\psi^{2}\\ 
+(C_{3} \rho (\mu+d) + C_{2}\mu (C_{3} + \phi))\psi^{3}.
\end{multline*}
Since all parameter values are positive, $(S_{n+1},I_{n+1}, C_{n+1},A_{n+1})$ 
is positive for all $n \geq 0$.
\end{proof}
 
 
\subsection{Conservation law}
\label{subsec:cons:law}

The second result to be shown is the conservation law or boundedness of the solutions.

\begin{theorem}
\label{conservation}
The NSFD scheme defines the discrete dynamical system 
\eqref{eq:discmodelexplicit} on 
\begin{equation}
\tilde{\Omega}=\left\{ (S_{n},I_{n},C_{n},A_{n}) 
: 0 \leq S_{n}+I_{n}+C_{n}+A_{n} \leq \frac{\Lambda}{\mu}\right\}.
\end{equation}
\end{theorem} 

\begin{proof} 
Let the total population be $N_{n}=S_{n}+I_{n}+C_{n}+A_{n}$. 
Adding the four equations of \eqref{eq:discmodel}, we have
\begin{equation*}
\dfrac{N_{n+1}-N_{n}}{\psi}
=\Lambda - \mu N_{n+1}-d A_{n+1} 
\Leftrightarrow (1+\mu \psi) N_{n+1}=\Lambda \psi+N_{n}-d \psi A_{n+1}\\
\end{equation*}
and
\begin{equation*}
N_{n+1} \leq \frac{\Lambda \psi}{1+\mu \psi}+ \dfrac{N_{n}}{1+\mu \psi}
\Leftrightarrow N_{n} \leq \Lambda \psi \sum_{j=1}^{n} 
\left(\frac{1}{1+\mu \psi}\right)^{j}+N_{0}\left(\frac{1}{1+\mu \psi}\right)^{n}.
\end{equation*}
By the discrete Grownwall inequality, if $0<N(0)< \frac{\Lambda}{\mu}$, then
\begin{equation}
N_{n} \leq \frac{\Lambda }{\mu}\left(  1- \frac{1}{(1+\mu \psi)^{n}}\right)
+N_{0}\left(  \frac{1}{1+\mu \psi}\right)^{n} =\dfrac{\Lambda}{\mu} 
+ \left( N_{0}- \frac{\Lambda}{\mu}\right)\left(\dfrac{1}{1+ \mu \psi}\right)^{n}
\end{equation}
and, since $\left( \dfrac{1}{1+ \mu \psi}\right) <1$, we have
$N_{n} \to \frac{\Lambda}{\mu}$ as $n \to \infty$. We conclude 
that the feasible region $\tilde{\Omega}$ 
is maintained within the discrete scheme.
\end{proof}


\subsection{Elementary stability}
\label{subsec:el:stab}

A difference scheme that approximates a first-order differential system 
is \emph{elementary stable} if, for any value of the step size, its 
fixed-points are exactly those of the differential system. Furthermore, 
when applied to the associated linearized differential system, 
the resulting difference scheme has the same stability/instability 
properties \cite{Dimitrov:06}.

The continuous and discrete system have the same equilibria. 
The disease free equilibrium (DFE) point is given by 
\begin{equation}
\label{eq:dfe}
E_{0}=(S^{\ast},I^{\ast},C^{\ast},A^{\ast})
=\left( \frac{\Lambda}{ \mu}, 0, 0, 0\right).
\end{equation}
The explicit expression of the endemic equilibrium point 
of \eqref{eq:discmodelexplicit} can be rewritten 
as the one of the continuous case: when 
\begin{equation*}
\lambda^{\ast}= \dfrac{\beta( I^{\ast}+\eta_{C} C^{\ast} + \eta_{A}
A^{\ast})}{N^{\ast}}=\dfrac{\mathcal{D}(\mathcal{R}_{0}-1)}{C_{2}C_{3}
+\phi C_{2}+\rho C_{3}},
\end{equation*}
we obtain the endemic equilibrium point. The explicit expression 
of the endemic equilibrium point of \eqref{eq:discmodelexplicit} 
can be rewritten as the one of the continuous case:
\begin{equation}
\label{eq:endemic}
S^{\ast}=\dfrac{\Lambda (\mathcal{D}
- \rho d C_{3})}{\mu (\mathcal{N} - \rho d C_{3})},
\quad
I^{\ast}=\dfrac{\Lambda C_{2} C_{3} 
(\mathcal{D}-\mathcal{N})}{\mathcal{D}(\rho d C_{3}-\mathcal{N})},
\quad
C^{\ast}=\dfrac{\Lambda C_{2} 
\phi (\mathcal{D}-\mathcal{N})}{\mathcal{D}(\rho d C_{3} 
-\mathcal{N})}, 
\quad A^{\ast}=\dfrac{\Lambda \rho C_{3} (\mathcal{D}
-\mathcal{N})}{\mathcal{D}(\rho d C_{3} -\mathcal{N})}.
\end{equation}
After some computations, we get the following equalities:
\begin{equation*}
\begin{split}
S^{\ast}&=\dfrac{\Lambda}{\lambda^{\ast}+\mu}
=\dfrac{\Lambda (C_{2}C_{3}+C_{2} \phi + \rho C_{3})}{ \mathcal{D}(\mathcal{R}_{0}-1)
+\mu(C_{2}C_{3}+C_{2} \phi + \rho C_{3})}=\dfrac{\Lambda (\mathcal{D}
-\rho d C_{3})}{\mu (\mathcal{N} - \rho d C_{3})},\\
I^{\ast}&=\dfrac{S^{\ast} \lambda^{\ast}C_{2}C_{3}}{\mathcal{D}}
= \dfrac{\Lambda (R_{0}-1)C_{2}C_{3}}{\mathcal{D}(\mathcal{R}_{0}-1)
+\mu(C_{2}C_{3}+C_{2} \phi + \rho C_{3})}
=\dfrac{\Lambda C_{2} C_{3} (\mathcal{D}-\mathcal{N})}{\mathcal{D}(\rho d C_{3}
-\mathcal{N})},\\
C^{\ast}&= \dfrac{S^{\ast} \lambda^{\ast} \phi C_{2}}{\mathcal{D}}
=\dfrac{\Lambda (R_{0}-1)C_{2}\phi}{\mathcal{D}(\mathcal{R}_{0}-1)
+\mu(C_{2}C_{3}+C_{2} \phi + \rho C_{3})}=\dfrac{\Lambda C_{2} \phi (\mathcal{D}
-\mathcal{N})}{\mathcal{D}(\rho d C_{3} -\mathcal{N})},\\
A^{\ast}&= \dfrac{S^{\ast} \lambda^{\ast} \rho C_{3}}{\mathcal{D}}
=\dfrac{\Lambda (R_{0}-1)\rho C_{3}}{\mathcal{D}(\mathcal{R}_{0}-1)
+\mu(C_{2}C_{3}+C_{2} \phi + \rho C_{3})}
=\dfrac{\Lambda \rho C_{3} (\mathcal{D}-\mathcal{N})}{\mathcal{D}(\rho d C_{3} 
-\mathcal{N})}.
\end{split}
\end{equation*}


\subsubsection{Local stability of the DFE point}

Let us discuss the stability of the proposed 
NSFD scheme at the DFE point $E_{0}$ \eqref{eq:dfe}. 

\begin{remark}
Several articles use the next-generation matrix
approach presented in \cite{Driessche}. 
For that, however, the matrices $F+T$ must be non-negative. 
Our model does not satisfy such condition.
\end{remark}

The following technical lemma has an important role in our proofs.

\begin{lemma}
\label{beta}
If $\mathcal{R}_{0}<1$, then $\beta$ must be smaller than $C_{1}$.
\end{lemma}

\begin{proof}
If $\mathcal{R}_{0}<1$, then 
\begin{align*}
\dfrac{\beta(C_{2} C_{3}+C_{3} \eta_{A} \rho
+C_{2} \eta_{C}\phi)}{C_{1} C_{2} C_{3} 
-C_{3} \alpha \rho -C_{2} \omega \phi}<1
&\Leftrightarrow
\beta(C_{2} C_{3}+C_{3} \eta_{A} \rho+C_{2} \eta_{C}\phi)
<C_{1} C_{2} C_{3} -C_{3} \alpha \rho -C_{2} \omega \phi \\
&\Leftrightarrow
\beta(C_{2} C_{3}+C_{3} \eta_{A} \rho+C_{2} \eta_{C}\phi)
-C_{1} C_{2} C_{3} -C_{3} \alpha \rho -C_{2} \omega \phi<0\\
&\Leftrightarrow
(\beta-C_{1}) C_{2}C_{3}<-\beta C_{3} \eta_{A} \rho
-\beta C_{2} \eta_{C}\phi -C_{3} \alpha \rho -C_{2} \omega \phi<0. 
\end{align*}
Since $C_{2}$ and $C_{3}$ are positive, 
we conclude that $\beta-C_{1} <0$.
\end{proof}

\begin{proposition}
\label{C1}
The first condition of Theorem~\ref{jury}, 
$p_{4}(1)>0$, is satisfied if $\mathcal{R}_{0}<1$.
\end{proposition}

\begin{proof}
The linearization of \eqref{eq:discmodelexplicit} at the DFE $E_{0}$ is:
\[
J(E_{0})=\left(
\begin{array}{cccc}
-\mu & -\beta  &-\beta \eta_{C} &-\beta\eta_{A}  \\
0&   \beta-(\rho+\phi+\mu)&\beta \eta_{c} + \omega & \beta \eta_{A}+\alpha \\
0&  \phi  & -(\omega + \mu)  & 0\\
0&  \rho & 0  & -(\alpha + \mu + d)
\end{array}
\right)=\left(
\begin{array}{cccc}
-\mu & -\beta  &-\beta \eta_{C} &-\beta\eta_{A}  \\
0&   \beta-C_{1}&\beta \eta_{c} + \omega & \beta \eta_{A}+\alpha \\
0&  \phi  & -C_{3}  & 0\\
0&  \rho & 0  & -C_{2}
\end{array}
\right).
\]
The characteristic polynomial of $J(E_{0})$ has the following expression:
\begin{equation*}
p_{4}(\lambda)=\lambda^{4}+p_{1}\lambda^{3}
+p_{2} \lambda^{2}+p_{3}\lambda +p_{4}=(-\mu-\lambda)p_{3}(\lambda),
\end{equation*}
where $p_{3}(\lambda)$ is given by
\begin{equation*}
\begin{split}
p_{3}(\lambda)&=-\lambda^{3}+(\beta-C_{1} -C_{2}-C_{3})\lambda^{2}
+((\beta-C_{1}) (C_{2}+C_{3}) -C_{2}C_{3}+(\beta \eta_{A}+\alpha)\rho 
+(\beta \eta_{C} + \omega) \phi)\lambda \\
&\quad +(-C_{1}C_{2}C_{3}+C_{2}C_{3}\beta +C_{3} \alpha \rho +C_{3} \beta 
\eta_{A} \rho +C_{2} \beta \eta_{C} \phi + C_{2} \phi \omega), 
\end{split}
\end{equation*}
that is,
\begin{equation*}
\begin{split}
p_{3}(\lambda)
&=-\lambda^{3}+(\beta-C_{1} -C_{2}-C_{3})\lambda^{2}
+\left((\beta-C_{1}) (C_{2}+C_{3}) -C_{2}C_{3}+(\beta \eta_{A}
+\alpha)\rho +(\beta \eta_{C} + \omega) \phi\right)
\lambda \\
&\quad +(\mathcal{N}-\mathcal{D}).
\end{split}
\end{equation*}
Since $p_{4}(1)>0$, we have $(-\mu -1)p_{3}(1)>0$, 
and recalling that $\mu>0$, we conclude that 
\begin{equation*}
p_{3}(1)=-1+(\beta -C_{1}-C_{2}-C_{3})+ (\beta-C_{1})(C_{2}+C_{3})
-C_{2}C_{3}+ \rho(\beta \eta_{A} + \alpha)
+ \phi(\beta \eta_{C}+\omega) - \mathcal{D}(1-\mathcal{R}_{0})
\end{equation*}
is negative: $p_{3}(1)<0$. If $\mathcal{R}_{0} <1$, then
\begin{equation*}
0<\rho(\beta \eta_{A} + \alpha)< -(\beta-C_{1})C_{2}
-\frac{\phi C_{2} (\beta \eta_{C} + \omega)}{C_{3}}
\end{equation*}
and
\begin{equation*}
(C_{1}-\beta)C_{2}-\frac{\phi C_{2} (\beta \eta_{C} + \omega)}{C_{3}}>0
\quad \textrm{ or }\quad 
\phi (\beta \eta_{C} + \omega)<C_{3}(C_{1}-\beta).
\end{equation*}
Therefore,
\begin{align*}
\rho(\beta \eta_{A} + \alpha)+\phi (\beta \eta_{C} 
+ \omega)<(C_{1}-\beta) C_{2}+\phi (\beta \eta_{C} 
+ \omega)\left( 1-\frac{C_{2}}{C_{3}}\right)
\end{align*}
and
\begin{align*}
p_{3}(1) 
&<-1+(\beta -C_{1}-C_{2}-C_{3})- (C_{1}-\beta)C_{3}-C_{2}C_{3}
- \mathcal{D}(1-\mathcal{R}_{0})+\phi (\beta \eta_{C} 
+ \omega)\left( 1-\frac{C_{2}}{C_{3}}\right)\\
&=-1+(\beta -C_{1}-C_{2}-C_{3})- (C_{1}-\beta)C_{3}-C_{2}C_{3}
- \mathcal{D}(1-\mathcal{R}_{0})+\phi (\beta \eta_{C} 
+ \omega)-\phi (\beta \eta_{C} + \omega)\left(\frac{C_{2}}{C_{3}}\right)\\
&<-1+(\beta -C_{1}-C_{2}-C_{3})-C_{2}C_{3}
- \mathcal{D}(1-\mathcal{R}_{0})-\phi (\beta \eta_{C} 
+ \omega)\left(\frac{C_{2}}{C_{3}}\right)< 0.
\end{align*}
We conclude that the first condition of Theorem~\ref{jury}
is satisfied if $\mathcal{R}_{0}<1$.
\end{proof}

\begin{proposition}
\label{C2}
If $\mathcal{R}_{0}<1$, $C_{2}<1$, $C_{3}<1$ and 
$\beta < \displaystyle \frac{C_{2} C_{3}}{(1-C_{2})(1-C_{3})}$, 
then the second condition of Theorem~\ref{jury}, that is, 
$(-1)^{4}p_{4}(-1)>0$, is satisfied.
\end{proposition}

\begin{proof}
Since $(-1)^{4} \cdot p_{4}(-1)>0$, we have
$(-\mu +1)p_{3}(-1)>0$ and $\mu<1$, so $p_{3}(-1)>0$. 
It can be seen that 
\begin{align*}
p_{3}(-1)&=1+\beta -C_{1}-C_{2}-C_{3}-(\beta-C_{1})(C_{2}+C_{3})
+C_{2}C_{3}- \rho(\beta \eta_{A} + \alpha)
- \phi(\beta \eta_{C}+\omega) + \mathcal{D}(\mathcal{R}_{0}-1).
\end{align*}
If $\mathcal{R}_{0}<1$, then
\begin{equation*}
-\rho(\beta \eta_{A} + \alpha)
> -(C_{1}-\beta)C_{2}+\frac{\phi C_{2} (\beta \eta_{C} + \omega)}{C_{3}}
\end{equation*}
and
\begin{equation*}
-(C_{1}-\beta)C_{2}+\frac{\phi C_{2} (\beta \eta_{C} + \omega)}{C_{3}}<0
\quad \textrm{ or } \quad 
-\phi (\beta \eta_{C} + \omega)>C_{3}(\beta-C_{1}).
\end{equation*}
Thus,
\begin{align}
\label{eq:condpart}
-\rho(\beta \eta_{A} + \alpha)-\phi (\beta \eta_{C} + \omega)
> -(C_{1}-\beta) C_{2}+\phi (\beta \eta_{C} 
+ \omega)\left( \frac{C_{2}}{C_{3}}-1 \right)
\end{align}
and
\begin{align*}
p_{3}&(-1) > 1+\beta -C_{1}-C_{2}-C_{3}
+ (C_{1}-\beta)C_{3}+C_{2}C_{3}
+\mathcal{D}(\mathcal{R}_{0}-1)+\phi (\beta \eta_{C} 
+ \omega)\left( \frac{C_{2}}{C_{3}}-1\right)\\
&=1+\beta -C_{1}-C_{2}-C_{3}+ (C_{1}-\beta)C_{3}+C_{2}C_{3}
+ \mathcal{D}(\mathcal{R}_{0}-1)-\phi (\beta \eta_{C} 
+ \omega)+\phi (\beta \eta_{C} + \omega)\left(\frac{C_{2}}{C_{3}}\right)\\
&>1+\beta -C_{1}-C_{2}-C_{3}+C_{2}C_{3}+\mathcal{D}(\mathcal{R}_{0}-1)
+\phi (\beta \eta_{C} + \omega)\left(\frac{C_{2}}{C_{3}}\right)\\
&=\phi (\beta \eta_{C} + \omega)\left(\frac{C_{2}}{C_{3}}\right)
+(1-C_{2})(1-C_{3})+\beta -C_{1}+\mathcal{D}(\mathcal{R}_{0}-1)\\
&=\phi (\beta \eta_{C} + \omega)\left(\frac{C_{2}}{C_{3}}\right)+(1-C_{2})(1-C_{3})
+\beta -C_{1}+(C_{1}C_{2}C_{3}-C_{3}\rho\alpha-C_{2}\phi\omega)(\mathcal{R}_{0}-1)\\
&=\phi (\beta \eta_{C} + \omega)\left(\frac{C_{2}}{C_{3}}\right)+C_{1}C_{2}C_{3}\mathcal{R}_{0}
+(C_{3}\rho\alpha+C_{2}\phi\omega)(1-\mathcal{R}_{0})+(1-C_{2})(1-C_{3})+\beta-C_{1}(1+C_{2}C_{3}).
\end{align*}
We conclude that if $C_{2}<1$, $C_{3}<1$, and 
$\beta < \displaystyle \frac{C_{2} C_{3}}{(1-C_{2})(1-C_{3})}$ 
are satisfied, then $p_{3}(-1)>0$.
\end{proof}

\begin{proposition}
\label{C3}
If $\mathcal{R}_{0}<1$, 
$\mu < \dfrac{1}{\mathcal{D}(1-\mathcal{R}_{0})}$, 
$p_{2}<1+p_{4}$, and
\begin{equation}
\label{eq:B3}
\dfrac{-(1-p_{4}^{2})(1+p_{2}+p_{4})}{(p_{1}+p_{3})}
<(p_{4}p_{1}-p_{3}) 
<\dfrac{(1-p_{4})^{2}(1+p_{4}-p_{2})}{(p_{1}-p_{3})},
\end{equation}
then the third condition of Theorem~\ref{jury} is satisfied.
\end{proposition}

\begin{proof}
The third condition of Theorem~\ref{jury} is the following:  
the  $3 \times 3$ matrices $B_{3}^{\pm}$ given by 
\begin{equation}
B_{3}^{\pm}=
\left(
\begin{array}{ccc}
1  & 0 & 0  \\
p_{1}   & 1 &0 \\
p_{2}& p_{1}&1 
\end{array}
\right)
\pm
\left(
\begin{array}{ccc}
0&0        &p_{4}   \\
0& p_{4} &p_{3}\\
p_{4}&p_{3}&p_{2}
\end{array}
\right)
\end{equation}
are positive innerwise. Recall that
$p_{4}(\lambda)=\lambda^{4}+p_{1}\lambda^{3}
+p_{2} \lambda^{2}+p_{3}\lambda +p_{4}=(-\mu-\lambda)p_{3}(\lambda)$ 
and
\begin{align*}
p_{1}&=\mu-(\beta-C_{1}-C_{2}-C_{3}),\\
p_{2}&=-\mu(\beta-C_{1}-C_{2}-C_{3}) -((\beta-C_{1}) (C_{2}+C_{3}) 
-C_{2}C_{3}+(\beta \eta_{A}+\alpha)\rho +(\beta \eta_{C} + \omega) \phi),\\
p_{3}&=-\mu((\beta-C_{1}) (C_{2}+C_{3}) -C_{2}C_{3}
+(\beta \eta_{A}+\alpha)\rho +(\beta \eta_{C} + \omega) \phi)-(\mathcal{N}-\mathcal{D}),\\
p_{4}&= -\mu(\mathcal{N}-\mathcal{D}),
\end{align*}
or
\begin{align*}
p_{1}&=C_{1}+C_{2}+C_{3}+\mu-\beta, \\
p_{2}&=\mu(C_{1}+C_{2}+C_{3}-\beta) +(C_{1}-\beta)(C_{2}+C_{3}) 
+C_{2}C_{3}-(\beta \eta_{A}+\alpha)\rho -(\beta \eta_{C} + \omega) \phi,\\
p_{3}&=\mu((C_{1}-\beta) (C_{2}+C_{3}) +C_{2}C_{3}
-(\beta \eta_{A}+\alpha)\rho -(\beta \eta_{C} 
+ \omega) \phi)+\mathcal{D}(1-\mathcal{R}_{0}),\\
p_{4}&= \mu\mathcal{D}(1-\mathcal{R}_{0}).
\end{align*}
Also, if $\mathcal{R}_{0}<1$, then, by Lemma~\ref{beta}, $\beta<C_{1}$. 
Therefore, $p_{1}>0$. If we also consider \eqref{eq:condpart}, 
and apply it to $p_{2}$ and $p_{3}$, we get
\begin{equation*}
p_{2}>\mu(C_{1} +C_{2}+C_{3} -\beta)+C_{2}C_{3}
+\frac{C_{2}\phi}{C_{3}}(\beta \eta_{C}+\omega)>0
\end{equation*}
and
\begin{equation*}
p_{3}>\mu C_{2}C_{3} + \mathcal{D}(1-\mathcal{R}_{0})>0.
\end{equation*}
In other words, $B_{3}^{+}$ and $B_{3}^{-}$ have the following form: 
\begin{equation}
B_{3}^{+}=\left(
\begin{array}{ccc}
1  & 0 & p_{4}  \\
p_{1}   & 1+p_{4} &p_{3}   \\
p_{2}+p_{4}& p_{1}+p_{3}&1+p_{2}
\end{array}
\right)
\quad \textrm{ and }\quad  
B_{3}^{-}=\left(
\begin{array}{ccc}
1  & 0 & -p_{4}  \\
p_{1}   & 1-p_{4} &-p_{3}   \\
p_{2}-p_{4}& p_{1}-p_{3}&1-p_{2}
\end{array}
\right)
\end{equation}
and their inners must be positive:
\begin{enumerate}
\item Regarding $B_{3}^{+}$, we must have
\begin{itemize}
\item[a)] $1+p_{4} >0$;
\item[b)] $|B_{3}^{+}|=(1+p_{4})(1+p_{2})-p_{3}(p_{1}
+p_{3})+p_{4}(p_{1} (p_{1}+p_{3})-(1+p_{4})(p_{2}+p_{4}))>0$.
\end{itemize}
\item Regarding $B_{3}^{-}$, we must have
\begin{itemize}
\item[a)] $1-p_{4} >0$;
\item[b)] $|B_{3}^{-}|=(1-p_{4})(1-p_{2})+p_{3}(p_{1}-p_{3})
-p_{4}(p_{1} (p_{1}-p_{3})-(1-p_{4})(p_{2}-p_{4}))>0$.
\end{itemize}
\end{enumerate}
Note that $p_{4}=\mu\mathcal{D}(1-\mathcal{R}_{0})=\det(J(E_{0}))$ 
and, if $\mathcal{R}_{0}<1$, then $p_{4}>0$. Therefore, 1~a) is satisfied. 
For 2~a) to be satisfied, it is necessary that
\begin{equation}
\label{eq:p4}
p_{4}<1 \Leftrightarrow \mu < \dfrac{1}{\mathcal{D}(1-\mathcal{R}_{0})}.
\end{equation}
Considering 1 b) and 2 b), after some computations, we can rewrite them as
\begin{gather*}
|B_{3}^{+}|=(1-p_{4}^{2})(1+p_{2}+p_{4})+(p_{1}+p_{3})(p_{4}p_{1}-p_{3})>0,\\
|B_{3}^{-}|=(1-p_{4})^{2}(1+p_{4}-p_{2})-(p_{1}-p_{3})(p_{4}p_{1}-p_{3})>0.
\end{gather*}
Thus, from \eqref{eq:p4}, $p_{2}<1+p_{4}$, and \eqref{eq:B3},  
the third condition of Theorem~\ref{jury} is satisfied.
\end{proof}

We are now in condition to prove the main result of this section.

\begin{theorem}
\label{local}
If $C_{2}<1$, $C_{3}<1$, $\beta < \frac{C_{2} C_{3}}{(1-C_{2})(1-C_{3})}$, 
$p_{2}<1+p_{4}$, and \eqref{eq:B3} and \eqref{eq:p4} are satisfied, 
then, provided $\mathcal{R}_{0} < 1$, the disease free equilibrium point 
of the discrete system \eqref{eq:discmodelexplicit} is locally asymptotically stable. 
If the previous conditions are not satisfied, then the disease free 
equilibrium point is unstable.
\end{theorem}

\begin{proof}
The result follows by Theorem~\ref{jury} 
and Propositions~\ref{C1}, \ref{C2} and \ref{C3}. 
If any of the conditions enumerated are not satisfied, 
at least one of the roots of the characteristic polynomial lies outside 
the unit circle, so the disease free equilibrium point is unstable.
\end{proof}


\subsubsection{Global stability of the equilibrium points}

Now we prove that $\mathcal{R}_{0}$ is a critical value for global stability: 
when $\mathcal{R}_{0}<1$,  the disease free equilibrium point 
is globally asymptotically stable; when $\mathcal{R}_{0}>1$, 
the endemic equilibrium is globally asymptotically stable.

\begin{theorem}
\label{globalDFE}
If $\mathcal{R}_{0} <1$, then the DFE point 
of the discrete-time SICA  model \eqref{eq:discmodel} 
is globally asymptotically stable.
\end{theorem}

\begin{proof}
For any $\varepsilon >0$, there exists an integer $n_{0}$ 
such that, for any $n\geq n_{0}$, $S_{n+1} < \frac{\Lambda}{\mu}+ \varepsilon$. 
Consider the sequence $\{ V(n)\}_{n=0}^{+ \infty}$ defined by
\begin{equation}
V(n)=I_{n}+\frac{\omega}{C_{3}}C_{n}+\frac{\alpha}{C_{2}}A_{n}+\psi\tilde{\lambda}_{n}S_{n+1}.
\end{equation}
For any $n \geq n_{0}$, 
\begin{align*}
V(n+1)&-V(n)=I_{n+1}+\frac{\omega}{C_{3}}C_{n+1}+\frac{\alpha}{C_{2}}A_{n+1}
+\psi\tilde{\lambda}_{n+1}S_{n+2} - I_{n}-\frac{\omega}{C_{3}}C_{n}
-\frac{\alpha}{C_{2}}A_{n}-\psi\tilde{\lambda}_{n}S_{n+1}\\
&=-\psi C_{1} I_{n+1} + \alpha \psi A_{n+1} + \omega \psi C_{n+1} 
+\frac{\omega}{C_{3}}C_{n+1} +\frac{\alpha}{C_{2}}A_{n+1}+ \psi
\tilde{\lambda}_{n+1}S_{n+2}-\frac{\omega}{C_{3}}C_{n}-\frac{\alpha}{C_{2}}A_{n}\\
&=\psi \tilde{\lambda}_{n+1}S_{n+2}+\frac{\omega}{C_{3}}(C_{n+1}-C_{n})
+\frac{\alpha}{C_{2}}(A_{n+1}-A_{n})-\psi C_{1} I_{n+1} 
+ \alpha \psi A_{n+1} + \omega \psi C_{n+1}\\
&=\psi\left( \tilde{\lambda}_{n+1}S_{n+2} + \frac{\omega}{C_{3}}(\phi I_{n+1}
-C_{3}C_{n+1})+\frac{\alpha}{C_{2}}(\rho I_{n+1} - C_{2} A_{n+1})-C_{1} I_{n+1} 
+ \alpha A_{n+1} + \omega C_{n+1}\right)\\
&= \psi \left(\tilde{\lambda}_{n+1}S_{n+2} + \left( \frac{\omega \phi}{C_{3}}
+\frac{\alpha \rho}{C_{2}}-C_{1}\right)I_{n+1}\right)\\
&=\psi \left(\tilde{\lambda}_{n+1}S_{n+2} + \left( C_{2} \omega \phi 
+ C_{3} \alpha \rho-C_{1} C_{2} C_{3}\right)\frac{I_{n+1}}{C_{2}C_{3}}\right)
=\psi\left( \tilde{\lambda}_{n+1}S_{n+2} - \mathcal{D} \frac{I_{n+1}}{C_{2}C_{3}}\right).
\end{align*}
Since
\begin{equation*}
\tilde{\lambda}_{n+1}S_{n+2} \leq \beta(I_{n+1}+\eta_{C}C_{n+1}
+\eta_{A} A_{n+1})\leq \beta \left(I_{n+1}+ \eta_{C} 
\frac{\phi I_{n+1}}{C_{3}}+ \eta_{A} \frac{\rho I_{n+1}}{C_{2}}\right),
\end{equation*}
we have 
\begin{align*}
V(n+1)-V(n)
&\leq \psi \left( \beta \left(I_{n+1}+ \eta_{C} \frac{\phi I_{n+1}}{C_{3}}
+ \eta_{A} \frac{\rho I_{n+1}}{C_{2}}\right) - \mathcal{D} \frac{I_{n+1}}{C_{2}C_{3}}  \right)\\
&=\frac{\psi I_{n+1}}{C_{2}C_{3}}\left(  \beta C_{2}C_{3} 
+\beta\eta_{C} C_{2}\phi +\beta \eta_{A} C_{3} \rho  -\mathcal{D} \right)\\
&=\frac{\psi I_{n+1}}{C_{2}C_{3}}\mathcal{D}(\mathcal{R}_{0}-1).
\end{align*}
If $\mathcal{R}_{0}<1$, and because $\varepsilon$ is arbitrary, 
we conclude that $V(n+1)-V(n) \leq 0$ and $\underset{n\to \infty}{\lim} I_{n}=0$ 
for any $n \geq 0$. The sequence $\{ V(n)\}_{n=0}^{+ \infty}$ 
is monotone decreasing and $\underset{n\to \infty}{\lim} S_{n}=\frac{\Lambda}{\mu}$.
\end{proof}

\begin{theorem}
\label{globalEE}
If $\mathcal{R}_{0}>1$, then the endemic equilibrium point 
of the discrete-time SICA model \eqref{eq:discmodel} 
is globally asymptotically stable.
\end{theorem}

\begin{proof}
We construct a sequence $\{ \tilde{V}(n)\}_{n=1}^{+\infty}$ of the form
\begin{equation*}
\tilde{V}(n)=\frac{1}{\psi I^{\ast}}g\left(\frac{S_{n}}{S^{\ast}}\right)
+\frac{1}{\psi S^{\ast}}g\left(\frac{I_{n}}{I^{\ast}}\right)
+\frac{\omega C^{\ast}}{\psi C_{3} S^{\ast} I^{\ast}}g\left(\frac{C_{n}}{C^{\ast}}\right)
+ \frac{\alpha A^{\ast}}{\psi C_{2} S^{\ast} I^{\ast}}g\left(\frac{A_{n}}{A^{\ast}}\right),
\end{equation*}
where $g(x)=x-1-\ln(x)$, $x \in \R^{+}$. Clearly, $g(x)\geq 0$ 
with equality holding true only if $x=1$. We have
\begin{align*}
g\left(\frac{S_{n+1}}{S^{\ast}}\right)-g\left(\frac{S_{n}}{S^{\ast}}\right)
&=\frac{S_{n+1}-S_{n}}{S^{\ast}}-\ln\left(\frac{S_{n+1}}{S_{n}}\right)
\leq\frac{(S_{n+1}-S^{\ast})(S_{n+1}-S_{n})}{S_{n+1}S^{\ast}}\\
&=\frac{S_{n+1}-S^{\ast}}{S_{n+1}S^{\ast}} \psi \left( \Lambda 
- \tilde{\lambda}_{n} S_{n+1} - \mu S_{n+1} \right)\\
&=\frac{S_{n+1}-S^{\ast}}{S_{n+1}S^{\ast}} \psi \left( \tilde{\lambda}^{\ast} S^{\ast} 
+ \mu S^{\ast}- \tilde{\lambda}_{n} S_{n+1} - \mu S_{n+1} \right)\\
&=-\frac{\mu \psi (S_{n+1}-S^{\ast})^{2}}{S_{n+1}S^{\ast}}
-\psi \tilde{\lambda}^{\ast} 
\left( 1- \frac{S^{\ast}}{S_{n+1}}\right)\left(
\frac{\tilde{\lambda}_{n}}{\tilde{\lambda}^{\ast}}\frac{S_{n+1}}{S^{\ast}} -1\right).
\end{align*}
Similarly,
\begin{align*}
g\left(\frac{I_{n+1}}{I^{\ast}}\right)-g\left(\frac{I_{n}}{I^{\ast}}\right)
&=\frac{I_{n+1}-I_{n}}{I^{\ast}}-\ln\left(\frac{I_{n+1}}{I_{n}}\right)
\leq\frac{(I_{n+1}-I^{\ast})(I_{n+1}-I_{n})}{I_{n+1}I^{\ast}}\\
&\leq \frac{(I_{n+1}-I^{\ast})}{I_{n+1}I^{\ast}}\left( \tilde{\lambda}_{n} 
S_{n+1} - C_{1}I_{n+1} + \alpha A_{n+1}  + \omega C_{n+1} \right)\\
&= \frac{(I_{n+1}-I^{\ast})}{I_{n+1}I^{\ast}} \left(  \tilde{\lambda}_{n} S_{n+1}  
- \frac{I_{n+1}  \tilde{\lambda}^{\ast}S^{\ast}}{I^{\ast}}-\frac{\alpha 
I_{n+1} A^{\ast}}{I^{\ast}}-\frac{\omega I_{n+1} C^{\ast}}{I^{\ast}}
+ \alpha A_{n+1}  + \omega C_{n+1} \right)\\
&= \left( 1-\frac{I^{\ast}}{I_{n+1}}\right)\left(\frac{\psi 
A^{\ast} \alpha}{I^{\ast}} \left( \frac{A_{n+1}}{A^{\ast}}
- \frac{I_{n+1}}{I^{\ast}}\right)+\frac{\psi \omega 
C^{\ast}}{I^{\ast}} \left( \frac{C_{n+1}}{C^{\ast}}
-\frac{I_{n+1}}{I^{\ast}}\right)\right)\\
&\quad + \left( 1-\frac{I^{\ast}}{I_{n+1}}\right)\frac{\psi  
S^{\ast}\tilde{\lambda}^{\ast}}{I^{\ast}} 
\left( \frac{\tilde{\lambda}_{n}}{\tilde{\lambda}^{\ast}}
\frac{S_{n+1}}{S^{\ast}} -\frac{I_{n+1}}{I^{\ast}}\right),
\end{align*}
\begin{align*}
g\left(\frac{C_{n+1}}{C^{\ast}}\right)-g\left(\frac{C_{n}}{C^{\ast}}\right)
&=\frac{C_{n+1}-C_{n}}{S^{\ast}}-\ln\left(\frac{C_{n+1}}{C_{n}}\right)
\leq\frac{(C_{n+1}-C^{\ast})(C_{n+1}-C_{n})}{C_{n+1}C^{\ast}}\\
&\leq \frac{(C_{n+1}-C^{\ast})}{C_{n+1}C^{\ast}} \left(\phi I_{n+1} 
- C_{3}C_{n+1} \right)=C_{3}\psi \left(1-\frac{C^{\ast}}{C_{n+1}}\right) 
\left( \frac{I_{n+1}}{I^{\ast}}-\frac{C_{n+1}}{C^{\ast}}\right),
\end{align*}
\begin{align*}
g\left(\frac{A_{n+1}}{A^{\ast}}\right)-g\left(\frac{A_{n}}{A^{\ast}}\right)
&=\frac{A_{n+1}-A_{n}}{A^{\ast}}-\ln\left(\frac{A_{n+1}}{A_{n}}\right)
\leq\frac{(A_{n+1}-A^{\ast})(A_{n+1}-A_{n})}{A_{n+1}A^{\ast}}\\
&\leq \frac{(A_{n+1}-A^{\ast})}{A_{n+1}I^{\ast}} \left( 
\rho  I_{n+1} - C_{2} A_{n+1} \right)=C_{2}\psi \left(1
-\frac{A^{\ast}}{A_{n+1}}\right) \left( 
\frac{I_{n+1}}{I^{\ast}}-\frac{A_{n+1}}{A^{\ast}}\right).
\end{align*}
The difference of $\tilde{V}(n)$ satisfies
\begin{align*}
\tilde{V}(n+1)-\tilde{V}(n)&=\frac{1}{\psi
I^{\ast}}\left(g\left(\frac{S_{n+1}}{S^{\ast}}\right)
-g\left(\frac{S_{n}}{S^{\ast}}\right)\right)
+\frac{1}{\psi S^{\ast}}\left(g\left(\frac{I_{n+1}}{I^{\ast}}\right)
-g\left(\frac{I_{n}}{I^{\ast}}\right)\right)\\
&\quad + \frac{\omega C^{\ast}}{\psi C_{3} S^{\ast}
I^{\ast}}\left(g\left(\frac{C_{n+1}}{C^{\ast}}\right)
-g\left(\frac{C_{n}}{C^{\ast}}\right)\right)
+ \frac{\alpha A^{\ast}}{\psi C_{2} S^{\ast} 
I^{\ast}}\left(g\left(\frac{A_{n+1}}{A^{\ast}}\right)
-g\left(\frac{A_{n}}{A^{\ast}}\right)\right)\\
&\leq -\frac{\mu (S_{n+1}-S^{\ast})^{2}}{I^{\ast} S^{\ast} S_{n+1}}
-\frac{\tilde{\lambda}^{\ast}}{I^{\ast}}\left(\frac{I^{\ast}}{I_{n+1}}
\frac{\tilde{\lambda}_{n}}{\tilde{\lambda}^{\ast}}\frac{S_{n+1}}{S^{\ast}}
-2-\frac{\tilde{\lambda}_{n}}{\tilde{\lambda}^{\ast}}
+\frac{S^{\ast}}{S_{n+1}}+\frac{I_{n+1}}{I^{\ast}}\right)\\
&\quad - \frac{\omega C^{\ast}}{S^{\ast} I^{\ast}}\left( 
\frac{I^{\ast} C_{n+1}}{I_{n+1} C^{\ast}} + \frac{C^{\ast} I_{n+1}}{C_{n+1} I^{\ast}} 
-2\right)-\frac{\alpha A^{\ast}}{S^{\ast} I^{\ast}} 
\left( \frac{I^{\ast} A_{n+1}}{I_{n+1} A^{\ast}} 
+\frac{A^{\ast}I_{n+1}}{A_{n+1} I^{\ast}}-2\right)\\
&\leq  -\frac{\mu (S_{n+1}-S^{\ast})^{2}}{I^{\ast} S^{\ast} S_{n+1}}
-\frac{\tilde{\lambda}^{\ast}}{I^{\ast}}\left( g\left( 
\frac{S^{\ast}}{S_{n+1}}\right)+g\left(\frac{I_{n+1}}{I^{\ast}} \right)
+ g\left( \frac{S_{n+1} \tilde{\lambda}_{n}I^{\ast}}{S^{\ast} 
\tilde{\lambda}^{\ast} I_{n+1}}\right)
-g\left(\frac{\tilde{\lambda}_{n}}{\tilde{\lambda}^{\ast}}\right)\right)\\
&\quad -\frac{\omega C^{\ast}}{S^{\ast} I^{\ast}} \left(  
g\left(\frac{I^{\ast} C_{n+1}}{I_{n+1} C^{\ast}} \right) 
+ g\left( \frac{C^{\ast} I_{n+1}}{C_{n+1} I^{\ast}}\right)\right)
-\frac{\alpha A^{\ast}}{S^{\ast} I^{\ast}}\left( 
g\left( \frac{I^{\ast} A_{n+1}}{I_{n+1} A^{\ast}}\right) 
+ g\left( \frac{A^{\ast} I_{n+1}}{A_{n+1} I^{\ast}}\right)\right).
\end{align*}
Therefore, $\{ \tilde{V}(n)\}_{n=1}^{+\infty}$ 
is a monotone decreasing sequence for any $n \geq 0$.
Since $ \tilde{V}(n) \geq 0$ and
$\underset{n \to \infty}{\lim}\left( \tilde{V}(n+1) -\tilde{V}(n)\right)=0$, 
we obtain that $\underset{n \to \infty}{\lim} S_{n+1}=S^{\ast}$, 
$\underset{n \to \infty}{\lim} I_{n+1}=I^{\ast}$, 
$\underset{n \to \infty}{\lim} C_{n+1}=C^{\ast}$ 
and $\underset{n \to \infty}{\lim} A_{n+1}=A^{\ast}$. 
This completes the proof.
\end{proof}


\section{Numerical simulations}
\label{sec4}

In this section, we apply our discrete model to a case study of Cape Verde 
\cite{MyID:455,MyID:359}. The data is the same of \cite{MyID:455} 
and the parameters too. We present here a resume of the information. 

Since the first diagnosis of AIDS in 1986, Cape Verde try to fight, 
prevent, and treat HIV/AIDS \cite{MyID:359,RCV}. In Table~\ref{dataCV}, 
the cumulative cases of infection by HIV and AIDS 
in Cape Verde from 1987 to 2014 is given.
\begin{table}[htp]
\setlength{\tabcolsep}{3mm}
\caption{Cumulative cases of infection by HIV/AIDS and the total population 
in Cape Verde in the period from 1987 to 2014 \cite{MyID:455,MyID:359,RCV}.}
\begin{center}
\begin{tabular}{cccccccc}\hline
\textit{Year} & \textit{1987} & \textit{1988} & \textit{1989} & \textit{1990} 
& \textit{1991} & \textit{1992} & \textit{1993}\\ \hline
\textit{HIV/AIDS} & 61&107&160&211&244&303&337\\
\textit{Population}&323972&328861&334473&341256&349326&358473&368423\\ \hline
\textit{Year} & \textit{1994} & \textit{1995} & \textit{1996} 
& \textit{1997} & \textit{1998} & \textit{1999} & \textit{2000}\\ \hline
\textit{HIV/AIDS}&358&395&432&471&560&660&779\\
\textit{Population}&378763&389156&399508&409805&419884&429576&438737\\ \hline
\textit{Year} & \textit{2001} & \textit{2002} & \textit{2003} 
& \textit{2004} & \textit{2005} & \textit{2006} & \textit{2007}\\ \hline
\textit{HIV/AIDS}&913&1064&1233&1493&1716&2015&2334\\
\textit{Population}&447357&455396&462675&468985&474224&478265&481278\\ \hline
\textit{Year} & \textit{2008} & \textit{2009} & \textit{2010} 
&\textit{2011} & \textit{2012} & \textit{2013} & \textit{2014}\\ \hline
\textit{HIV/AIDS}&2610&2929&3340&3739&4090&4537&4946\\
\textit{Population}&483824&486673&490379&495159&500870&507258&513906\\ \hline
\end{tabular}
\end{center}
\label{dataCV}
\end{table}
Based on \cite{RCV,WBD}, the values for the initial 
conditions are taken as
\begin{equation}
S_{0}=323911,
\quad I_{0}=61,
\quad C_{0}=0,
\quad A_{0}=0.
\end{equation}
Regarding the parameter values, we consider $\rho=0.1$ \cite{SGS} 
and $\gamma=0.33$ \cite{BGW}. It is assumed that, after one year, 
HIV infected individuals, $I$, that are under ART treatment, 
have low viral load \cite{PEC} and are transferred to class $C$, 
so that $\phi=1$. The ART treatment therapy takes a few years. 
Following \cite{MyID:359}, it is assumed the default treatment rate 
to be 11 years ($1 / \omega$ years, to be precise). 
Based in \cite{ZE}, the induced death rate by AIDS is $d=1$. 
From the World Bank data \cite{WBD,WBDP}, the natural rate 
is assumed to be $\mu=1/69.54$. The recruitment rate 
$\Lambda=13045$ was estimated in order to approximate the values 
of the total population of Cape Verde, see Table~\ref{dataCV}. 
Based on a research study known as HPTN 052, 
where it was found that the risk of HIV transmission among heterosexual 
serodiscordant is $96\%$ lower when the HIV-positive partner is on treatment \cite{CC},  
we take here $\eta_{C}=0.04$, which means that HIV infected individuals 
under ART treatment have a very low probability 
of transmitting HIV \cite{Romero}. For the parameter  
$\eta_{A} \geq 1$, which accounts the relative infectiousness of individuals 
with AIDS symptoms, in comparison to those infected with HIV with 
no AIDS symptoms, we assume, based on \cite{WLG}, that $\eta_{A}=1.35$. 
For $(\eta_{C}, \eta_{A})=(0.04,1.35)$, the estimated value 
of the HIV transmission rate is equal to $\beta = 0.695$.
Using these parameter values, the basic reproduction number 
is $\mathcal{R}_{0}=4.5304$ and the endemic equilibrium point 
$(S^{\ast}, I^{\ast}, C^{\ast}, A^{\ast})=(145276, 48136.4, 461146, 3580.57)$.
In Figure~\ref{fig1}, we show graphically the cumulative cases of infection 
by HIV/AIDS in Cape Verde given in Table~\ref{dataCV}, together with the curves 
obtained from the continuous-time SICA model \eqref{eq:model}
and our discrete-time SICA model \eqref{eq:discmodel}.
Our simulations of the continuous and discrete models 
were done with the help of the software Wolfram \textsf{Mathematica}, version 12.1. 
For the continuous model, we have used the command \texttt{NSolve}, 
that computes the solution by interpolation functions. 
Our implementation for the discrete case makes use of the 
\textsf{Mathematica} command \texttt{RecurrenceTable}.

\begin{figure}[H]
\centering
\includegraphics[scale=0.7]{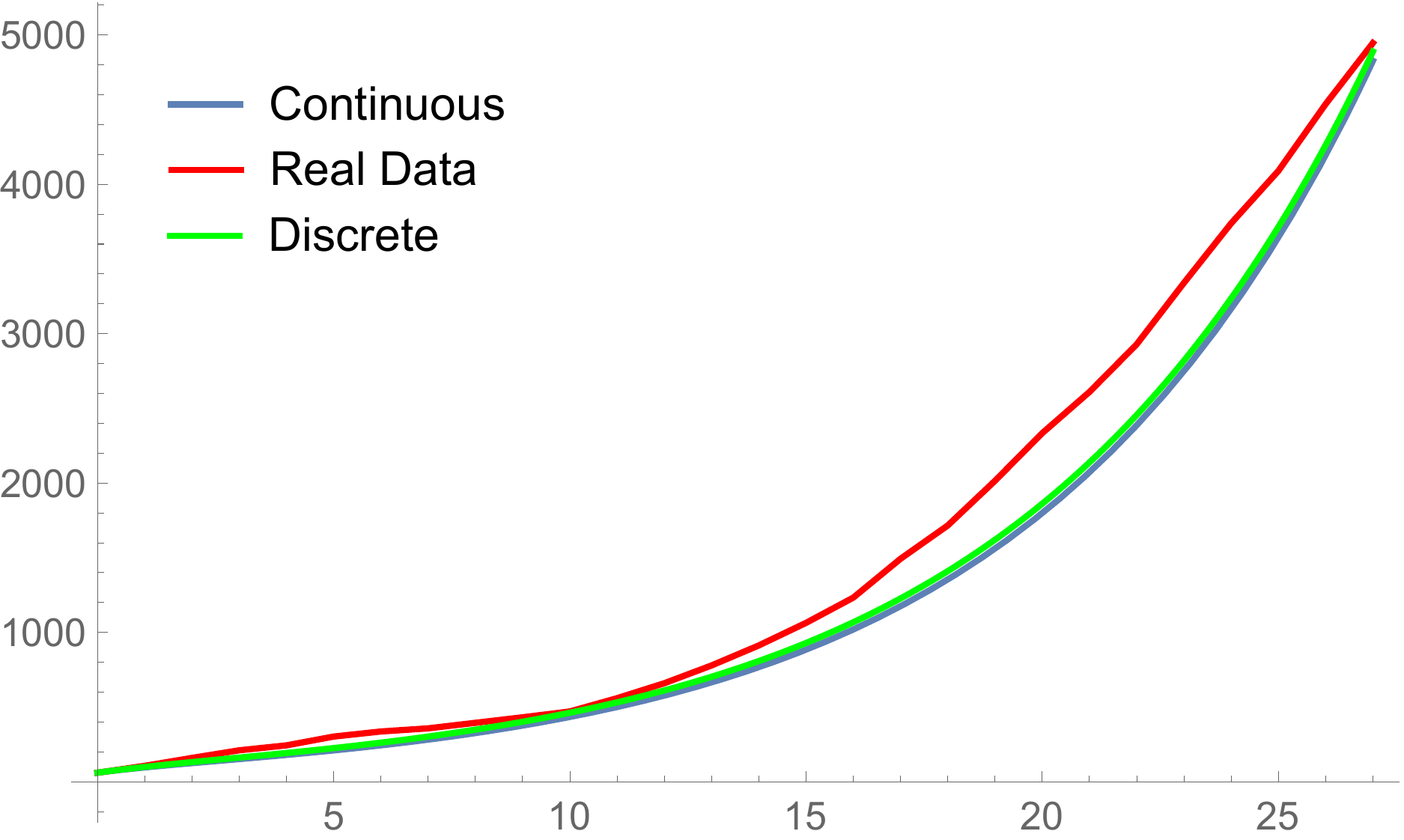}
\caption{Cumulative cases of infection by HIV/AIDS in Cape Verde 
in the period from 1987 (year 0) to 2014 (year 27): real data (red); 
prediction from the continuous-time SICA model \eqref{eq:model} (blue); 
and prediction from our discrete-time SICA model \eqref{eq:discmodel} (green).}
\label{fig1}
\end{figure}


To illustrate the global stability of the endemic equilibrium (EE),
predicted by our Theorem~\ref{globalEE}, we consider different 
initial conditions borrowed from \cite{MyID:455}, 
from different regions of the plane:
\begin{equation}
\label{eq:init:cond:NS}
\begin{split}
\text{(SD1,ID1,CD1,AD1)}&=(S_{0},I_{0},C_{0},A_{0}),\\
\text{(SD2,ID2,CD2,AD2)}&=(S_{0}/2,I_{0} +S_{0}/2,C_{0}+10^{4},A_{0}+4 \times10^{4}),\\
\text{(SD3,ID3,CD3,AD3)}&=(S_{0}/3,I_{0},C_{0}+4 \times 10^{4},A_{0}+S_{0}/3),\\
\text{(SD4,ID4,CD4,AD4)}&=(3 S_{0}/2,I_{0}+S_{0}/4,C_{0}+5 \times 10^{5},A_{0}+S_{0}/5).
\end{split}
\end{equation}
The obtained results are given in Figure~\ref{fig2}.
\begin{figure}[H]
\centering
\begin{subfigure}{7.5cm}
\includegraphics[scale=0.5]{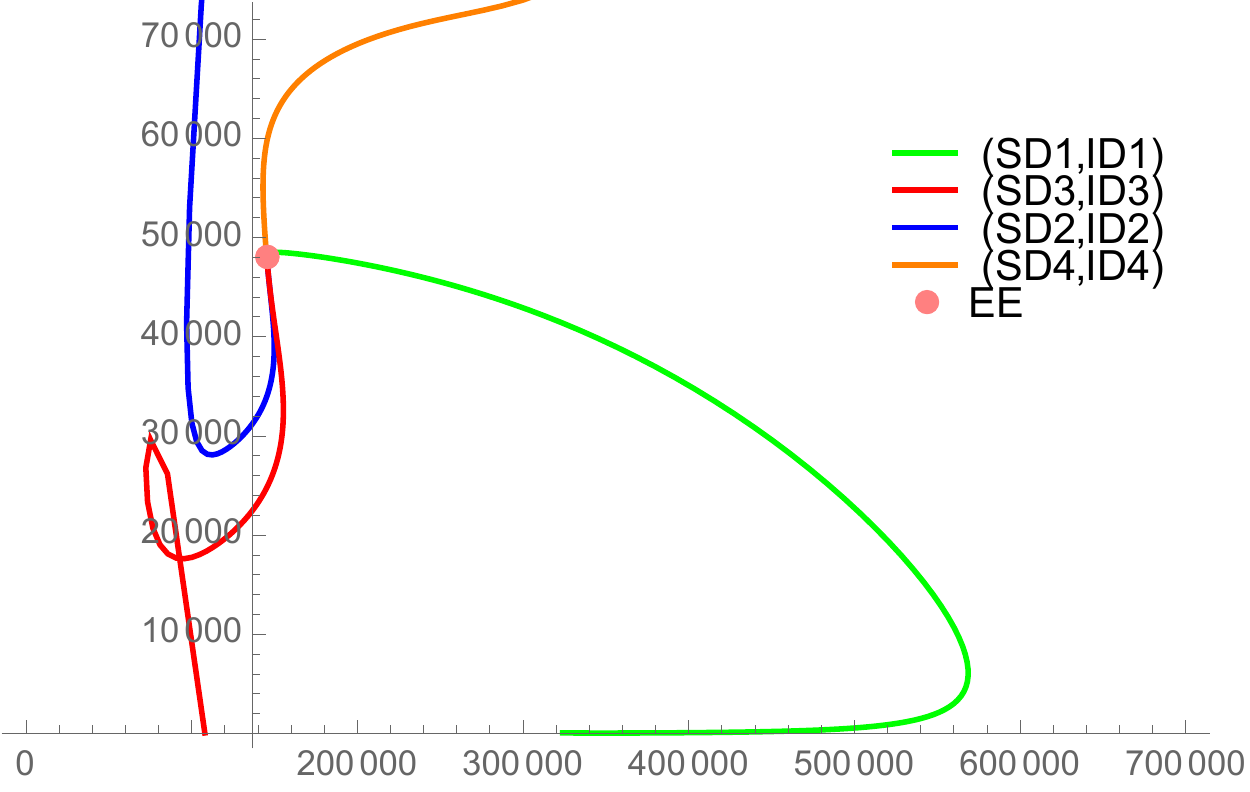}
\caption{HIV-infected individuals with no clinical symptoms ($I$)
versus Susceptible individuals ($S$).}
\end{subfigure}
\quad
\begin{subfigure}{7.5cm}
\includegraphics[scale=0.35]{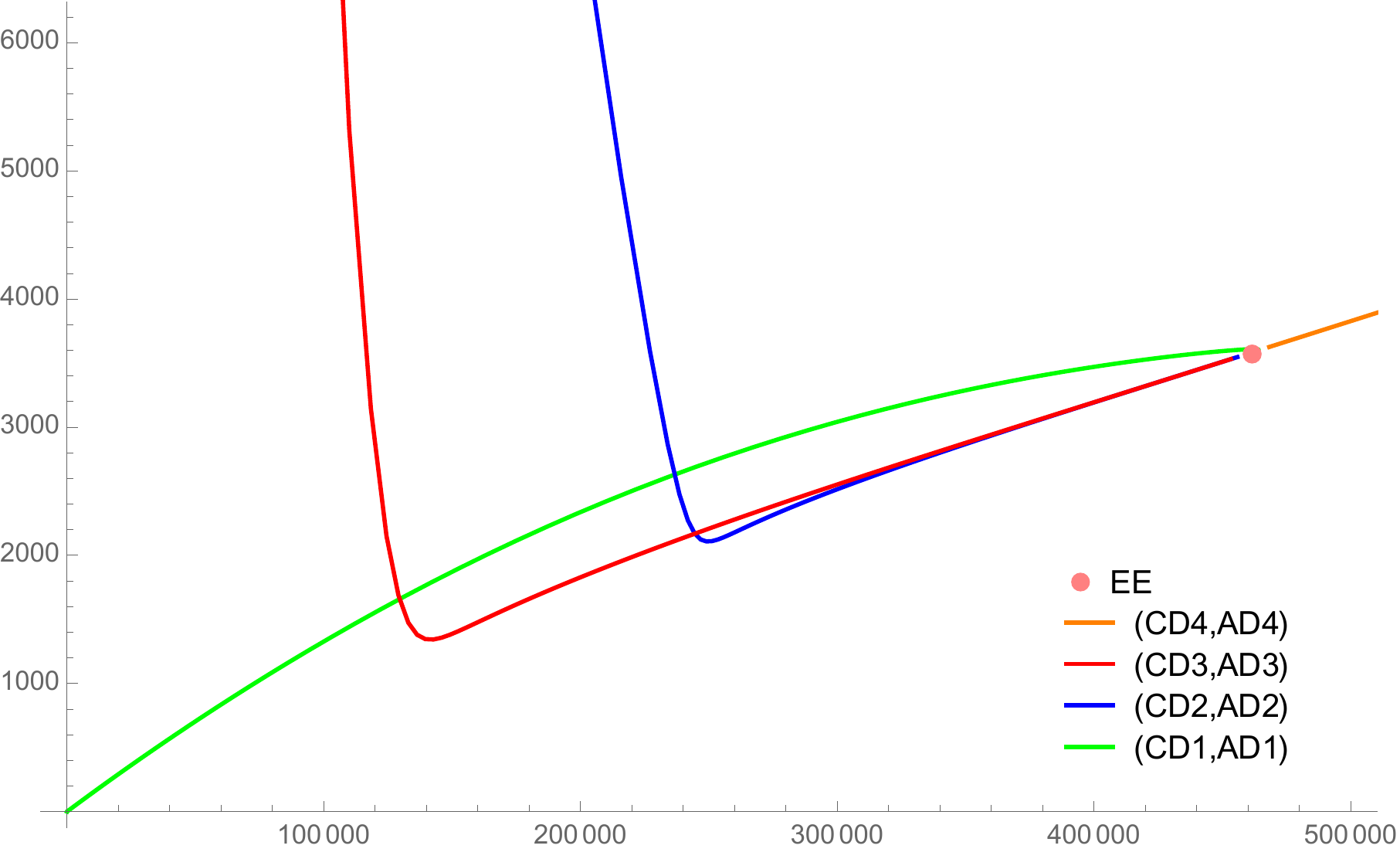}
\caption{HIV-infected individuals with AIDS ($A$)
versus HIV-infected individuals under treatment ($C$).}
\end{subfigure}
\caption{Illustration of the fact that the endemic equilibrium point 
$(S^{\ast}, I^{\ast}, C^{\ast}, A^{\ast})=(145276, 48136.4, 461146, 3580.57)$
of the discrete-time SICA model \eqref{eq:discmodel} 
is globally asymptotically stable. This is in agreement with
Theorem~\ref{globalEE}, since $\mathcal{R}_{0}=4.5304 > 1$.}
\label{fig2}
\end{figure}


\section{Conclusions}
\label{sec5}

In this work, we proposed a discrete-time SICA model, 
using Mickens' nonstandard finite difference 
(NSFD) scheme. The elementary 
stability was studied and the global stability of the equilibrium points proved. 
Finally, we made some numerical simulations, comparing our discrete model with 
the continuous one. For that, we have used the same data, following
the case study of Cape Verde. Our conclusion is that the discrete model 
can be used with success to describe the reality of Cape Verde,
as well as to properly approximate the continuous model. 
All our simulations have been done using 
the numerical computing environment \emph{Mathematica}, version 12.1, 
running on an Apple MacBook Pro i5 2.5 GHz with 16Gb of RAM. 
The solutions of the models were found in ``real time''.

Mickens was a pioneer in NSFD schemes. Throughout the years, 
other NSFD schemes were developed. Roughly speaking, different 
Mickens-type methods differ on the denominator functions
and the discretization, depending on concrete conditions 
that the continuous model under study must satisfy.
In \cite{MR4141413}, the incidence rate is combined, 
while in \cite{MR3093413} all parameters are constant. 
Other types of NSFD are presented, e.g., in 
\cite{MR3575285,MR3316778,MR2854820,MR2316130},
which can be used if the system satisfy some conditions 
and a suitable denominator function is constructed. For such schemes, 
the incidence functions are different from ours. 
In \cite{MR3316778}, for example, it is fundamental to rewrite the system 
and the discretization method and the denominator function are different from the ones
we use here. The article \cite{MR2854820} uses the same approach of 
\cite{MR3316778} and the model has a bilinear incidence function. 
Positive and elementary stable nonstandard
finite-difference methods are also considered in \cite{MR2316130}.
Mickens set the field, but several different authors developed 
and are developing other related discretization methods.  
For the SICA model, however, as we have shown, a new NSFD scheme is not necessary 
and the standard Mickens' method provides a well posed discrete-time
model with excellent results, without the need to impose 
additional conditions to the model. 


\section*{Acknowledgments}

The authors were partially supported by 
the Portuguese Foundation for Science and Technology (FCT):
Sandra Vaz through the Center of Mathematics and Applications 
of \emph{Universidade da Beira Interior} (CMA-UBI), 
project UIDB/00212/2020; Delfim F. M. Torres through
the Center for Research and Development in Mathematics 
and Applications (CIDMA), project UIDB/04106/2020.
They are very grateful to three anonymous referees, 
who kindly reviewed an earlier version of this manuscript 
and provided valuable suggestions and comments.


\section*{Conflict of interest}

The authors declare that they have no conflicts of interest.



\end{document}